\documentclass[reqno,a4paper]{amsart}

\usepackage{graphicx} 
\usepackage[a4paper,margin=1.25in]{geometry}
\usepackage{amsfonts,amsthm,amssymb,epsfig,amsmath}  
\usepackage{tikz}
\usetikzlibrary{arrows.meta,calc,intersections}
\usepackage{enumitem}
\usepackage{framed}
\usepackage{tikz-cd}

\DeclareMathOperator{\Fix}{Fix}
\DeclareMathOperator{\Centr}{Centr}
\DeclareMathOperator{\Crit}{Crit}

\DeclareMathOperator{\Diff}{Diff}
\DeclareMathOperator{\SL}{SL}
\DeclareMathOperator{\SO}{SO}
\DeclareMathOperator{\id}{id}

\newcommand{\s}{\mathbb{S}}
\newcommand{\N}{\mathbb{N}}
\newcommand{\Q}{\mathbb{Q}}
\newcommand{\R}{\mathbb{R}}
\newcommand{\RP}{\mathbb{R}\mathbb{P}}
\newcommand{\Z}{\mathbb{Z}}

\newcommand{\Sn}{\mathbb{S}^n}
\newcommand{\E}{\mathcal{E}}

\newtheorem{thm}{Theorem}

\newtheorem{theorem}{Theorem} 
\numberwithin{theorem}{section}
\numberwithin{equation}{section}
\newtheorem{lemma}[theorem]{Lemma}     
\newtheorem{corollary}[theorem]{Corollary}
\newtheorem{proposition}[theorem]{Proposition}

\theoremstyle{definition}

\numberwithin{example}{section}

\title[Groups of diffeomorphisms with prescribed critical regularity]{Finitely generated groups of diffeomorphisms with prescribed critical regularity in every dimension}
\author{Maximiliano Escayola}
\address{School of Mathematics, Korea Institute for Advanced Study (KIAS), Seoul 02455, Korea}
\email{maxiescayola@kias.re.kr}

\author{Carl-Fredrik Nyberg-Brodda}
\address{June E Huh Center for Mathematical Challenges, Korea Institute for Advanced Study (KIAS), Seoul 02455, Korea}
\email{cfnb@kias.re.kr}

\date{\today}


\thanks{The first author is supported by KIAS Individual Grant MG107601 at Korea Institute for Advanced Study (KIAS). The second author is supported by KIAS Individual Grant HP094701 at the June E Huh Center for Mathematical Challenges of KIAS. Both authors are supported by the Mid-Career Researcher Program (RS-2023-00278510) through the National Research Foundation funded by the government of Korea.}

\begin{document}

\begin{abstract}
Let $n\ge 2$, and let $M$ be either the $n$-dimensional sphere $\s^n$ or the real projective space $\RP^n$. Given $\alpha\ge 1$, we construct a finitely generated group $G_{n,\alpha}$ that embeds into the group of $C^\alpha$-diffeomorphisms of $M$, but does not embed into the group of $C^\beta$-diffeomorphisms of $M$ for any $\beta>\alpha$.  
\end{abstract}

\maketitle

\section{Introduction}

\noindent Let $M$ be a compact, connected manifold, and let $\alpha\in\{0\}\cup[1,\infty]$. This work is concerned with understanding the local algebraic structure of the group $\Diff^\alpha(M)$, and how this structure changes as the differentiability parameter $\alpha$ varies. Here, if $\alpha\notin\N$, $\Diff^\alpha(M)$ denotes the group of $C^{\lfloor \alpha\rfloor}$-diffeomorphisms of $M$ whose derivatives of order $\lfloor \alpha\rfloor$ are ($\alpha-\lfloor \alpha\rfloor$)-Hölder continuous; for integer $\alpha$, we use the usual $C^\alpha$-regularity. When $M$ is oriented, we denote by $\Diff_+^\alpha(M)$  the subgroup of orientation-preserving elements of $\Diff^\alpha(M)$. With this convention, $\Diff^0(M)$ denotes the group of homeomorphisms of $M$, and $\Diff_+^0(M)$ its orientation-preserving subgroup. 

It was shown by Whittaker~\cite{whittaker} that the global algebraic structure of $\Diff^0(M)$ detects the underlying manifold: that is, $\Diff^0(M)\cong\Diff^0(N)$ implies that $M$ and $N$ are homeomorphic. Rubin \cite{Rubin} generalized this to homeomorphism groups of other topological spaces. In his work~\cite{filipkiewicz}, Filipkiewicz showed that, given $\alpha,\beta\in\N$, $\Diff^\alpha(M)\cong\Diff^\beta(N)$ implies $\alpha=\beta$, and there exists a $C^\alpha$-diffeomorphism between $M$ and $N$ conjugating the two diffeomorphism groups. More recently, Kim, Koberda, and de la Nuez González~\cite{kim-koberda-delanuez} considerably strengthened these reconstruction results by showing that the first-order theory of the diffeomorphism group already detects both the manifold and the regularity. Recall that two groups $G$ and $H$ are said to be elementarily equivalent, and we write $G\equiv H,$ if they satisfy the same first-order sentences in the language of groups. They proved that, for $\alpha,\beta\in\{ 0 \}\cup [1,\infty]$, the equivalence $\Diff^\alpha(M)\equiv\Diff^\beta(N)$ implies that $\alpha=\beta$ and $M$ and $N$ are $C^\alpha$-diffeomorphic. 

In their pioneering work~\cite{KimKoberda}, Kim and Koberda showed that, when $M$ is the circle $\R/\Z$ or a compact interval, the differentiability class can already be detected at the level of finitely generated groups. More precisely, for every $\alpha\ge 1$, there exists a finitely generated group $G_\alpha$ whose critical regularity on $M$ is equal to $\alpha$. In this context, we define the \emph{critical regularity} of a group $G$ on $M$ by 
\[
\Crit_M(G)=\alpha
\]
if $G$ embeds into $\Diff^\alpha(M)$ but does not embed into $\Diff^\beta(M)$ for any $\beta>\alpha$. Mann and Wolff~\cite{mann-wolff} later gave a different approach to the problem of constructing groups of prescribed critical regularity, based on reconstructing individual homeomorphisms from the algebraic structure of finitely generated groups containing them. This allowed them to relate critical regularity to differentiable rigidity and, in particular, to obtain an independent proof of the one-dimensional existence result.

In this note, we construct groups of prescribed critical regularity on the $n$-dimensional sphere~$\s^n$ and on the real projective space~$\RP^n$, for every $n\ge2$. 
\begin{thm}\label{thm: main} Let $n\ge2,$ and let $M$ be either $\Sn$ or  $\RP^n$. For every $\alpha\ge1$, there exists a finitely generated group $G_{n,\alpha}$ such that $\Crit_{M}(G_{n,\alpha})=\alpha$.
\end{thm}

A key ingredient in the proof of Theorem~\ref{thm: main} is the rigidity theorem of Brown, Rodriguez Hertz, and Wang for lattice actions on spheres and real projective spaces. Since comparable rigidity results are not presently available for general manifolds, we do not know how to extend our construction beyond this setting. To the best of our knowledge, these are the first examples, for manifolds of dimension greater than one, of finitely generated groups whose critical regularity is known to be finite and positive.

\subsection{Outline of the construction} Before beginning the construction, we would like to briefly explain the main ideas behind it. Although we do not directly use their approach, the underlying idea arose from trying to implement a technique proposed by Mann and Wolff \cite{mann-wolff} for constructing finitely generated groups of prescribed critical regularity, in any differentiability class and on any smooth manifold $M$. Roughly speaking, given $\alpha\ge1$, a group of critical regularity equal to $\alpha$ can be obtained by finding a group $\Gamma\le \Diff^\infty(M)$ with the following properties:
\begin{enumerate}
    \item $C^\beta$-rigidity: every injective group homomorphism $\Gamma\to\Diff^\beta(M)$ comes from conjugation by a $C^\beta$-diffeomorphism of $M$.
    \item Map recognition: there exists some element $f\in\Diff^\alpha(M)\setminus\bigcup_{\beta>\alpha}\Diff^\beta(M)$ such that, for every $h\in\Diff^\alpha(M)$, the existence of a group isomorphism $\phi:\langle\Gamma, f\rangle\to\langle\Gamma,h\rangle$ with $\phi|_\Gamma=\text{id}_\Gamma$ and $\phi(f)=h$ implies $f=h$.
\end{enumerate}
Notice that the two conditions above easily imply $\Crit_M(\langle\Gamma,f\rangle)=\alpha$. Indeed, suppose that there exists an embedding $\phi \colon \langle\Gamma,f\rangle\to\Diff^\beta(M)$ with $\beta>\alpha$. By the rigidity property of $\Gamma$, after normalizing the embedding, that is, conjugating $\phi$ by a suitable element, we may assume that its restriction to $\Gamma$ is the identity. The map recognition property then implies that $\phi(f)=f$, contradicting the fact that $f$ is not of class $C^\beta$.

Let $n\ge2$. The construction for $\RP^n$ follows the same strategy as that for $\s^n$, so we describe only the latter. Our construction starts with a rigid copy $\Gamma\simeq\SL_{n+1}(\Z)$ inside $\Diff_+^{\infty}(\s^n)$, which we enlarge by adjoining finitely many elements of a subgroup $K\simeq \SO(n)$ preserving latitudes and acting transitively on each of them. We finally add a copy $\hat{H}_\alpha$ of the one-dimensional group $H_\alpha$ provided by Kim and Koberda (see Theorem \ref{thm: kim-koberda} below). The additional elements are chosen so that their squares generate a dense subgroup of $K$, while $\hat{H}_\alpha$ commutes with $K$.

To prove that the resulting group, say $G_{n,\alpha}$, has critical regularity $\alpha$, we suppose that it admits an embedding into $\Diff^\beta(\s^n)$ for some $\beta>\alpha$. Rigidity allows us to normalize the embedding so that it is the identity on $\Gamma$. A recognition property for the chosen elements of $K$, together with density, then forces the image of $\hat{H}_\alpha$ to centralize $K$. Passing to the orbit space $\s^n/K$, which is naturally an interval, produces an action of $[H_\alpha,H_\alpha]$ by $C^\beta$ interval diffeomorphisms. The simplicity of $[H_\alpha,H_\alpha]$ and the structure of the kernel of this quotient action imply that this action is faithful, contradicting the fact that $[H_\alpha,H_\alpha]$ has critical regularity $\alpha$ on the interval.

\section{Preliminaries}

\noindent In this section, we collect the main ingredients that will be used in the construction and proof of Theorem~\ref{thm: main}. We begin with the one-dimensional examples of Kim and Koberda, and then recall the rigidity result of Brown, Rodriguez Hertz, and Wang for lattice actions on spheres and real projective spaces.  
\subsection{Critical regularity in dimension one}

The following result is one of the main theorems of \cite{KimKoberda} (see Theorem 1.4 and Theorem A.3 therein). We state it in a slightly different form, better suited to our purposes. For the statement, we denote by $\Diff_c^\alpha(-1,1)$ the subgroup of compactly supported elements of $\Diff_+^\alpha(-1,1)$.
\begin{theorem}[Kim and Koberda]\label{thm: kim-koberda}
    Let $I$ be a compact interval. For every $\alpha\ge 1$, there exists a finitely generated group $H_\alpha\le\Diff_+^\alpha(I)$ such that:
    \begin{enumerate}
        \item Its commutator subgroup $[H_\alpha,H_\alpha]$ is infinite, nonabelian, and simple.
        \item $[H_\alpha,H_\alpha]$ does not embed into $\Diff_+^\beta(I)$ for any $\beta>\alpha$.
    \end{enumerate} 
    Moreover, $\Diff_+^\alpha(I)$  may be replaced by $\Diff_c^\alpha(-1,1)$.
\end{theorem}

\subsection{Lattices and projective actions of $\SL$}\label{Subsec:lattices-background}
Recall that, for $n \geq 2$, the standard projective action $\rho$ of $\SL_{n+1}(\R)$ on $\mathbb{S}^{n}$ is defined by  
\begin{equation}\label{eq:representation of SL(Z)}
\rho(g)(u) = \frac{gu}{\lVert gu \rVert}, \quad  g \in \SL_{n+1}(\R),\quad u \in \mathbb{S}^{n}.
\end{equation}
It is well known that this gives an injective homomorphism  $\rho \colon \SL_{n+1}(\R) \to \Diff^\infty_+(\mathbb{S}^{n})$. For completeness, we include a proof in the following lemma.
\begin{lemma}
The map $\rho$ is an injective homomorphism into $\Diff_+^\infty(\s^n)$.    
\end{lemma}
\begin{proof}
 Given $g,h\in\SL_{n+1}(\R)$, it is straightforward to verify that $\rho(g)(\rho(h)(u))=\rho(gh)(u)$. The $C^\infty$-regularity of $\rho(g)$ follows from the fact that $gu\neq\vec{0}$ whenever $u\neq\vec{0}$. We now prove the faithfulness of the action. Suppose that $\rho(g)=\id_{\s^n}$ for some $g\in\SL_{n+1}(\R)$. Then, $gu=\lVert gu\rVert u$ for all $u\in\s^n$. Therefore, whenever $u,w\in\s^n$ are linearly independent, we have
 \[
 \lVert gu\rVert= \frac{\lVert g(u+w)\rVert}{\lVert u+w\rVert}=\lVert gw\rVert.
 \]
 Thus, every nonzero vector is an eigenvector with the same eigenvalue, and hence $g=\lambda I$ for some $\lambda>0.$ Since $\det(g)=\lambda^{
 n+1}=1,$ we obtain $\lambda=1$, and therefore $g=I$. Finally, since $\SL_{n+1}(\R)$ is connected, every $g\in \SL_{n+1}(\R)$ can be joined to the identity by a path. Applying $\rho$ to this path gives an isotopy between $\rho(g)$ and the identity. Hence, $\rho(g)$ is orientation-preserving. 
\end{proof}

Let $\pi \colon \Sn \to \RP^n$ be the quotient map identifying antipodal points. It is the universal two-sheeted covering of $\RP^n$. We denote by $\tilde{\rho} \colon \SL_{n+1}(\R) \to \Diff^\infty(\RP^n)$ the standard projective action, given by $\tilde{\rho}(\gamma)([v])=[\gamma v]$. The actions of $\rho$ and $\tilde{\rho}$ are related by the following commutative diagram  
\[
\begin{tikzcd}
\mathbb S^n \arrow[r, "\rho(\gamma)"] \arrow[d, "\pi"'] &
\mathbb S^n \arrow[d, "\pi"] \\
\mathbb{RP}^n \arrow[r, "\tilde{\rho}(\gamma)"'] &
\mathbb{RP}^n
,\end{tikzcd}
\]
that is,
\begin{equation}\label{Eq:projective-actions-commute}
\pi \circ \rho(\gamma) = \tilde{\rho}(\gamma) \circ \pi. 
\end{equation}
Unlike the standard projective action of $\SL_{n+1}(\R)$ on $\Sn$, the action $\tilde{\rho}$ is faithful only if $n$ is even; if $n$ is odd, then the kernel consists of $\pm I$.  

The relevance of the two actions $\rho$ and $\tilde\rho$ for our purposes comes from a recent rigidity theorem of Brown, Rodriguez Hertz, and Wang \cite[Theorem~1.1]{brown-rodriguez-wang}, which shows that, in this dimension, actions with infinite image are essentially conjugate to the standard projective action. We recall the precise statement below. 

\begin{theorem}[Brown, Rodriguez Hertz, and Wang]\label{Thm:BrownEtAl}
Let $n \geq 2$, let $\Lambda$ be a lattice in $\SL_{n+1}(\R)$, and let $M$ be a connected compact manifold with dimension $n$. Fix $\beta > 1$, and let $\sigma \colon \Lambda \to \Diff^\beta(M)$ have infinite image. Then:
\begin{enumerate}[label=(\roman*)]
    \item there is a $C^\beta$-diffeomorphism $h$ between $M$ and either $\mathbb{S}^{n}$ or $\RP^{n}$, 
    \item there is a subgroup $\Lambda' \leq \Lambda$ with
    \begin{enumerate}
        \item $\Lambda' = \Lambda$ if $M \simeq \RP^{n}$, or 
        \item $[\Lambda : \Lambda'] \leq 2$ if $M \simeq \mathbb{S}^{n}$, 
    \end{enumerate}
    such that 
    \item for every $x \in M$ and every $\gamma \in \Lambda'$, we have $h(\sigma(\gamma)(x)) = \varrho(\gamma)(h(x))$, where $\varrho$ denotes the corresponding standard projective action, namely $\rho$ on $\s^n$ or $\tilde\rho$ on $\RP^{n}$. 
\end{enumerate}
\end{theorem}

We will apply this theorem to the lattice $\Lambda=\SL_{n+1}(\Z)$, both when $M=\Sn$ and when $M = \RP^n$. In the first case, the finite-index ambiguity in Theorem~\ref{Thm:BrownEtAl} disappears, and we obtain the following direct consequence. Indeed, $\s^n\not\simeq\RP^n$, and, since $\SL_{n+1}(\Z)$ is perfect for $n\ge2$ (a classical result of Nielsen \cite{Nielsen1924}, cf.\ also Steinberg \cite{Steinberg1968}) it has no subgroup of index 2.

\begin{corollary}\label{Cor:theusefulform}
Let $n \geq 2$ and $\beta>1$, and suppose that $\sigma \colon \SL_{n+1}(\Z) \to \Diff^\beta(\mathbb{S}^{n})$ has infinite image. Then $\sigma$ is $C^\beta$-conjugate to the standard projective action $\rho$. 
\end{corollary}
In the case of $M = \RP^n$ rather than the sphere, we obtain the following result. 
\begin{corollary}\label{Cor:theusefulform-projectivespace}
Let $n \geq 2$ and $\beta>1$, and suppose that $\sigma \colon \SL_{n+1}(\Z) \to \Diff^\beta(\RP^n)$ has infinite image. Then $\sigma$ is $C^\beta$-conjugate to the standard projective action $\tilde\rho$. 
\end{corollary}

Note that this corollary requires no perfectness, since Theorem~\ref{Thm:BrownEtAl}(ii)(a) ensures that we can take all of $\Lambda = \SL_{n+1}(\Z)$ as our $\Lambda'$.

\subsection{Rigidity and map recognition for the standard lattice actions}
Fix $n\ge2$, and define the groups
\begin{equation}\label{eq: def Gamma and tilde Gamma}
\Gamma:=\rho(\SL_{n+1}(\Z)),\quad\tilde\Gamma:=\tilde\rho(\SL_{n+1}(\Z)).
\end{equation}
We now turn to the map-recognition properties that will be used in the proof of our main result. The precise statements are given in Proposition~\ref{Prop:Recognition-prop}. Their proof relies on an elementary lemma concerning the centralizers of finite-index subgroups of $\Gamma$ and $\tilde\Gamma$. For the statement, let $A$ denote the antipodal map on $\Sn$, i.e. the map $A(v) = -v$. The following lemma can also be obtained from Furman's alignment machinery~\cite{furman} by applying them to the projectivization $\R^{n+1}\setminus\{0\}\to\RP^n$, writing this map as the radial projection onto $\Sn$ followed by $\pi:\Sn\to\RP^n$, and using the uniqueness of maps $\Sn\to\RP^n$ compatible with the $H$-actions. Since there is also a direct, slightly longer, elementary proof, we include the latter for completeness.

\begin{lemma}\label{Lem:centralizer-of-P(SL)}
Let $n \geq 2$ and let $H$ be a finite index subgroup of $\SL_{n+1}(\Z)$. Then,
\begin{enumerate}
    \item\label{item: 1 lem centr} $\Centr_{\Diff^0(\Sn)}(\rho(H)) = \{ \id_{\s^n}, A \},$
    \item\label{item: 2 lem centr} $\Centr_{\Diff^0(\RP^n)}(\tilde\rho(H)) = \{\id_{\RP^n}\}.$
\end{enumerate}
\end{lemma}
\begin{proof}
We begin with the proof of \eqref{item: 1 lem centr}. For ease of notation, we consider $\Sn$ as a subspace of $\R^{n+1}$, and accordingly use the language of rational hyperplanes, lines, and great spheres. First, for $i \neq j$, let $T_{ij} = I + E_{ij}$. For $k \geq 1$, we have $T_{ij}^k = I + kE_{ij}$. It is easy to see that the fixed-point set of $\rho (T_{ij}^k)$ on $\Sn$ is $ \Sn \cap \{ x_j = 0 \}$. Furthermore, every rational hyperplane $\mathcal{H}$ is obtained as the $\SL_{n+1}(\Z)$-translate of one of the hyperplanes $ \Sn \cap \{ x_j = 0 \}$. Hence, it follows that there exists some $h \in H$ such that $\Fix(\rho(h)) = \mathcal{H} \cap \Sn$. Every rational line $L\subset\R^{n+1} $ can be written as the intersection of $n$ rational hyperplanes, say $\mathcal{H}_1,\ldots,\mathcal{H}_n$. By the preceding argument, for each $k=1,\ldots,n$, there exists $h_k\in H$ such that $\Fix(\rho(h_k))=\mathcal{H}_k\cap\s^n$. Therefore, 
\[
L\cap \s^n=\bigcap_{k=1}^n\Fix(\rho(h_k)).
\]

Now let $f \in \text{Centr}_{\Diff^0(\Sn)}(\rho(H))$. Since $f$ commutes with $\rho(h_k)$, it preserves $\Fix(\rho(h_k))$ for every $k$. Hence, $f$ preserves $L\cap\s^n$ for every rational line. In particular, if $u\in \s^n$ lies in a rational line, then $f(u) \in \{ u, -u \}$. But unit vectors lying on rational lines are dense in $\Sn$, and the sign function $u \mapsto \langle f(u), u \rangle \in \{ -1, 1 \}$ is continuous on $\Sn$, so it is constant. Therefore, $f$ is either $\id_{\s^n}$ or the antipodal map $A$.

We now prove \eqref{item: 2 lem centr}. Let $f \in \Centr_{\Diff^0(\RP^n)}(\tilde\rho(H))$, so that $f \tilde\rho(\gamma) f^{-1} = \tilde\rho(\gamma)$ for all $\gamma \in H$. We first lift $f$ to $\Sn$. Consider $f \circ \pi \colon \Sn \to \RP^n$. Because $\Sn$ is the universal covering space of $\RP^n$ with deck transformation group $\{ \id_{\s^n} , A\}$, where $A$ is the antipodal map, we can lift $f$ to a map $F \colon \Sn \to \Sn$ such that $\pi \circ F = f \circ \pi$ which is a homeomorphism since $f$ is a homeomorphism. There are two such lifts: we fix one and denote it by $F$, so that the other is $A \circ F$. 

Recall that the projective actions $\rho, \tilde\rho$ of $\SL_{n+1}(\Z)$ on $\Sn$ resp.\ $\RP^n$ satisfy the relation \eqref{Eq:projective-actions-commute}. We now claim that for every $\gamma \in H$, we have that $F \rho(\gamma) F^{-1}$ and $\rho(\gamma)$ are lifts of the same map $\Sn \to \RP^n$. Indeed, we have 
\begin{align*}
    \pi \circ  F \rho(\gamma) F^{-1} = f \tilde\rho(\gamma) f^{-1} \circ \pi= \tilde\rho(\gamma) \circ \pi= \pi \circ \rho(\gamma),
\end{align*}
where we have used \eqref{Eq:projective-actions-commute}, the universal property $\pi \circ F = f \circ \pi$, and $f \tilde\rho(\gamma) f^{-1} = \tilde\rho(\gamma)$. Hence the two maps are lifts of the same map $\tilde\rho(\gamma) \circ \pi \colon \Sn \to \RP^n$. Since $\Sn$ is connected, two lifts of the same map differ by a unique deck transformation, and hence for each $\gamma \in H$ there exists a unique $\varepsilon(\gamma) \in \Z / 2\Z$ such that $F \rho(\gamma) F^{-1} = A^{\varepsilon(\gamma)} \rho(\gamma)$. 

The map $\varepsilon \colon H \to \Z / 2\Z$ is a homomorphism since $A$ commutes with every element of $\rho(H)$. Let $H_0 = \ker(\varepsilon)$. Then $H_0$ has finite index in $H$, and hence also in $\SL_{n+1}(\Z)$. Since now $F \in \Centr_{\Diff^0(\Sn)}(\rho(H_0))$, we can apply item \eqref{item: 1 lem centr} to conclude that $F \in \{\id_{\s^n} , A \}$. Both $\id_{\s^n}$ and $A$ project to the identity map on $\RP^n$, since $\pi(x) = \pi(-x)$, so that $\pi \circ F = \pi$, and hence $f \circ \pi = \pi$. Since $\pi \colon \Sn \to \RP^n$ is surjective, it follows that $f = \id_{\RP^n}$, completing the proof of the lemma.
\end{proof}
 \begin{proposition}[Rigidity and map recognition]\label{Prop:Recognition-prop}
Let $\Gamma\le\Diff_+^\infty(\s^n)$, $\tilde\Gamma\le\Diff^\infty(\RP^n)$ be the groups defined in \eqref{eq: def Gamma and tilde Gamma}. Then the following hold:
\begin{enumerate}
    \item\label{item: 1 prop rig-map r} For every $\beta>1$, every embedding of $\Gamma\to\Diff^\beta(\s^n)$ or $\tilde\Gamma\to \Diff^\beta(\RP^n)$ is induced by conjugation by a $C^\beta$-diffeomorphism.
    \item\label{item: 2 prop rig-map r} Let $R \in \rho(\SL_{n+1}(\Q))$, and let $h \in \Diff^0(\Sn)$. If there exists a homomorphism $\varphi \colon \langle \Gamma, R \rangle \to \langle \Gamma, h \rangle$ such that $\varphi |_\Gamma = \id_\Gamma$ and $\varphi(R) = h$, then $R^2  = h^2$.
    \item\label{item: 3 prop rig-map r-projective} Let $R \in \tilde\rho(\SL_{n+1}(\Q))$, and let $h \in \Diff^0(\RP^n)$. If there exists a homomorphism $\varphi \colon \langle \tilde\Gamma, R \rangle \to \langle \tilde\Gamma, h \rangle$ such that $\varphi |_{\tilde\Gamma} = \id_{\tilde\Gamma}$ and $\varphi(R) = h$, then $R  = h$.
\end{enumerate}
\end{proposition}
\begin{proof}
The proof of \eqref{item: 1 prop rig-map r} is a straightforward consequence of Corollaries~\ref{Cor:theusefulform} and~\ref{Cor:theusefulform-projectivespace}, and is left to the reader. We begin by proving~\eqref{item: 2 prop rig-map r}. Recall the following fact about commensurators inside $\SL_{n+1}(\Z)$: if $M \in \SL_{n+1}(\Q)$, then the group $H:=\SL_{n+1}(\Z) \cap M^{-1} \SL_{n+1}(\Z) M$ has finite index in $\SL_{n+1}(\Z)$. Indeed, if $d \in \N_{\geq 1}$ is such that $dM$ and $dM^{-1}$ have integral entries, then $H$ is easily seen to contain the principal congruence subgroup $\{S\in \SL_{n+1}(\Z): S\equiv I\pmod{d^2}\}$, and hence has finite index in $\SL_{n+1}(\Z)$. 

Let now $\Gamma_R := \Gamma \cap R^{-1} \Gamma R$. Then by the aforementioned commensurator fact, and by our assumptions on $R$, we have that $\Gamma_R$ has finite index in $\Gamma$. For $\gamma \in \Gamma_R$, we have $R \gamma R^{-1} \in \Gamma$. If we apply $\varphi$ to this, we get $h \gamma h^{-1} = R\gamma R^{-1}$, because $\varphi$ fixes $\Gamma$ by assumption. Consequently $(R^{-1}h)\gamma(R^{-1}h)^{-1} = \gamma$. Therefore, $R^{-1}h \in \Centr_{\Diff^0(\Sn)}(\Gamma_R)$. But by Lemma~\ref{Lem:centralizer-of-P(SL)}, the centralizer of any finite index subgroup of $\Gamma$ is $\{ \id_{\s^n}, A\}$, where $A$ is the antipodal map. Hence $R^{-1}h \in \{ \id_{\s^n}, A\}$, so that either $h = R$ or else $h = RA$. Since $A$ commutes with $R$, and since $A^2 = \id_{\s^n}$, we have $R^2=h^2$ in both cases, as required. 

We now prove \eqref{item: 3 prop rig-map r-projective}. Choose $M \in \SL_{n+1}(\Q)$ such that $\tilde\rho(M)=R$, and put $H = \SL_{n+1}(\Z) \cap M^{-1} \SL_{n+1}(\Z)M$. Then exactly as in the previous item the group $H$ has finite index in $\SL_{n+1}(\Z)$. Furthermore, for $\gamma \in H$, we have $R \tilde\rho(\gamma) R^{-1} \in \tilde\Gamma$. Applying $\varphi$, we obtain $h \tilde\rho(\gamma) h^{-1} = R \tilde\rho(\gamma)R^{-1}$. Hence $R^{-1} h$ centralizes $\tilde \rho(H)$. But by Lemma~\ref{Lem:centralizer-of-P(SL)}, the centralizer of $\tilde\rho(H)$ is trivial. Hence $R = h$, as desired. 
\end{proof}

\section{Constructing the groups}\label{sec:Constructing the groups}

\noindent Fix $n\ge2$ and $\alpha \geq1$. Let $\rho:\SL_{n+1}(\R)\to\Diff_+^\infty(\s^n)$ and $\tilde\rho:\SL_{n+1}(\R)\to\Diff^\infty(\RP^n)$ be the standard actions introduced in \eqref{eq:representation of SL(Z)} and \eqref{Eq:projective-actions-commute}, respectively. Let $\Gamma$ and $\tilde\Gamma$ denote the corresponding lattice images of $\SL_{n+1}(\Z)$ defined in \eqref{eq: def Gamma and tilde Gamma}. We now construct the groups that will realize the prescribed critical regularity on $\s^n$ and $\RP^n$. The two constructions are closely related and are obtained by adjoining suitable elements to $\Gamma$ and $\tilde\Gamma$, respectively.

\subsection{Groups on the sphere}\label{Subsec:sphere group of crit reg}
We start by enlarging $\Gamma$ as follows. Define 
\begin{equation}\label{eq: def of subg K isom SO(n)}
K:=\rho\left\lbrace \begin{pmatrix}
    B & \vec{0}^t\\
    \vec{0} & 1
\end{pmatrix} : B\in \SO(n) \right\rbrace\simeq \SO(n).
\end{equation}
As part of the enlargement of $\Gamma$, we choose elements $R_1,\ldots,R_{n-1}\in K$ arising from matrices in $\SL_{n+1}(\Q)$ such that the group generated by $R_1^2,\ldots,R_{n-1}^2$ is dense in $K$. The following elementary lemma records that this choice can indeed be made.
\begin{lemma}\label{Lem:exists-dense-rotations}
    There exist elements $R_1,\ldots,R_{n-1}\in K$, arising from matrices in $\SL_{n+1}(\Q)$, such that, for all $m\ge 1$, the subgroup $\langle R_1^m,\ldots,R_{n-1}^m \rangle$ is dense in $K$.
\end{lemma}
\begin{proof}
    Choose an angle $\theta$ such that $\tan\theta=4/3$. For $i=1,\ldots,n-1$, let $r_i\in\SO(n)$ be the rotation by $\theta$ in the plane spanned by $e_i,e_{i+1}$, acting trivially on its orthogonal complement. Since $\cos\theta=3/5$ and $\sin\theta=4/5$, each $r_i$ has rational coefficients. It is easy to verify that $\theta/2\pi\notin\Q$. Consequently, the powers of $r_i^m$ are dense in the circle subgroup of rotations in this plane.

    Since every element of $\SO(n)$ can be written as a product of rotations in the adjacent coordinate planes, approximating each factor by a power of the corresponding $r_i^m$ shows that $\overline{\langle r_1^m,\ldots,r_{n-1}^m\rangle}=\SO(n)$. Taking $R_i\in K$ to be the image under $\rho$ of 
    diag$(r_i,1)$ gives the result.
    \end{proof}

In addition, we adjoin one further subgroup. Let $H_\alpha\le\Diff_c^\alpha(-1,1)$ be the group provided by Theorem~\ref{thm: kim-koberda}, and consider its image, say $\hat{H}_\alpha$, under the embedding  
\begin{equation}\label{eq: embedding of H_alpha}
H_\alpha\to\Diff^0_+(\s^n),\quad g\mapsto\hat{g},
\end{equation}
obtained by extending the map
\begin{equation}\label{eq: hat g}
\hat{g}(v,t):=\left(\sqrt{\frac{1-g(t)^2}{1-t^2}}\, v,g(t)\right) 
\end{equation}
to the whole sphere $\s^n=\{(v,t)\in \R^n\times\R:\Vert v\Vert^2+|t|^2=1\}$ (see Figure~\ref{fig:action}). 

\begin{figure}[h]
    \centering
    \begin{tikzpicture}[scale=0.8,
  line cap=round,
  line join=round,
  >=Latex,
  font=\small,
  sphere fill/.style={fill=black!8},
  sphere outline/.style={line width=.78pt},
  latitude front/.style={line width=.65pt},
  latitude back/.style={line width=.48pt,dashed,dash pattern=on 3pt off 2.6pt},
  fixed meridian/.style={line width=.82pt},
  action arrow/.style={-{Latex[length=2.1mm,width=1.25mm]},line width=.82pt},
  interval/.style={line width=.68pt},
  guide/.style={line width=.42pt,densely dotted},
  clean label/.style={fill=white,inner sep=1pt}
]

\def\R{2.85}   
\def\flat{0.19}
\def\yt{1.30}  
\def\yq{0.45}  
\def\M{1.02}   
\def\Ix{-4.55} 
\def\Ih{2.85}  

\pgfmathsetmacro{\at}{sqrt(\R*\R-\yt*\yt)}
\pgfmathsetmacro{\bt}{\flat*\at}
\pgfmathsetmacro{\aq}{sqrt(\R*\R-\yq*\yq)}
\pgfmathsetmacro{\bq}{\flat*\aq}
\pgfmathsetmacro{\be}{\flat*\R}


\draw[guide] (\Ix,\yt) -- (-\at,\yt);
\draw[guide] (\Ix,\yq) -- (-\aq,\yq);
\draw[guide] (\Ix,-\Ih) -- (0,-\Ih);
\draw[guide] (\Ix,\Ih) -- (0,\Ih);

\draw[interval] (\Ix,-\Ih) -- (\Ix,\Ih);
\draw[interval,fill=white] (\Ix,\Ih) circle[radius=.073];
\draw[interval,fill=white] (\Ix,-\Ih) circle[radius=.073];
\node[above=4pt] at (\Ix,\Ih) {$1$};
\node[below=4pt] at (\Ix,-\Ih) {$-1$};

\fill (\Ix,\yt) circle[radius=.068];
\fill (\Ix,\yq) circle[radius=.068];
\node[left=7pt] at (\Ix,\yt) {$t$};
\node[left=7pt] at (\Ix,\yq) {$g(t)$};
\draw[action arrow] (\Ix,{\yt-0.28}) -- (\Ix,{\yq+0.28}) node[midway,right=5pt] {$g$};

\fill[sphere fill] (0,0) circle[radius=\R];
\draw[sphere outline] (0,0) circle[radius=\R];

\draw[latitude back] (0,\yt) ++(0:\at)
  arc[start angle=0,end angle=180,x radius=\at,y radius=\bt];
\draw[latitude back] (0,\yq) ++(0:\aq)
  arc[start angle=0,end angle=180,x radius=\aq,y radius=\bq];

\path[name path=latitude-t]
  (0,\yt) ++(180:\at)
  arc[start angle=180,end angle=360,x radius=\at,y radius=\bt];
\path[name path=latitude-q]
  (0,\yq) ++(180:\aq)
  arc[start angle=180,end angle=360,x radius=\aq,y radius=\bq];

\draw[latitude front] (0,\yt) ++(180:\at)
  arc[start angle=180,end angle=360,x radius=\at,y radius=\bt];
\draw[latitude front] (0,\yq) ++(180:\aq)
  arc[start angle=180,end angle=360,x radius=\aq,y radius=\bq];

\path[name path=fixed-geodesic]
  plot[domain=-90:90,samples=180] ({\M*cos(\x)},{\R*sin(\x)});
\draw[fixed meridian]
  plot[domain=-90:90,samples=180] ({\M*cos(\x)},{\R*sin(\x)});

\path[name intersections={of=fixed-geodesic and latitude-t,by=Xt}];
\path[name intersections={of=fixed-geodesic and latitude-q,by=Xq}];

\draw[action arrow] (Xt) to[bend left=3] (Xq);
\fill (Xt) circle[radius=.068];
\fill (Xq) circle[radius=.068];
\node[above right=1pt] at (Xt) {$(v,t)$};
\node[left=3pt] at ($(Xt)!0.52!(Xq)$) {$\hat g$};

\coordinate (FormulaAnchor) at (3.45,-.50);
\draw[guide] (Xq) -- ($(FormulaAnchor)+(-.15,0)$);
\node[anchor=west,clean label] at (FormulaAnchor)
  {$\displaystyle
    \hat g(v,t)=
    \left(\sqrt{\frac{1-g(t)^2}{1-t^2}}\,v,\,g(t)\right)$};

\draw[sphere outline,fill=white]
    (0,\R) circle[radius=.073];
\draw[sphere outline,fill=white]
    (0,-\R) circle[radius=.073];


\end{tikzpicture}
    \caption{The extension $g \mapsto \hat{g}$ for the Kim-Koberda group $H_\alpha \leq \Diff_c^\alpha(-1,1)$ to the sphere $\Sn$.}
    \label{fig:action}
\end{figure}
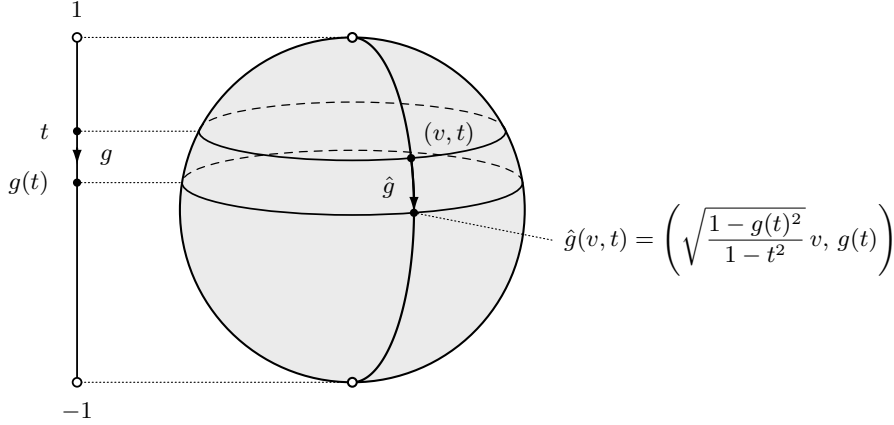

The following lemma shows that $\hat{H}_\alpha$ is an embedded copy of $H_\alpha$ in $\Diff_+^\alpha(\s^n)$.

\begin{lemma}\label{lem: hat H_alpha}
The map $H_\alpha\to\Diff^0_+(\s^n),\quad g\mapsto\hat{g},$ defined by \eqref{eq: hat g} is an injective homomorphism. Moreover, for every $g\in H_\alpha$, the homeomorphism $\hat{g}$ belongs to $\Diff_+^\alpha(\s^n)$.
\end{lemma}
\begin{proof}
    A straightforward application of the formula \eqref{eq: hat g} yields
    \[
    \hat{f}(\hat{g}(v,t))=\widehat{fg}(v,t),
    \]
    so the map $H_\alpha\to\Diff^0_+(\s^n)$, $g\to\hat{g}$, is a homomorphism. Injectivity follows by comparing the last coordinates: indeed the last coordinate of $\hat{g}(v,t)$ is $g(t)$, and hence $\hat{g}$ uniquely determines~$g$.

    We now verify that every element $\hat{g}\in\hat{H}_\alpha$ is of class $C^\alpha$. Since the projection of $\hat{g}$ onto the second coordinate is $g$, it is enough to check that the extension to the poles of the map 
    \begin{equation}\label{eq: projection 1st coordinate hat g}
    (v,t)\mapsto \sqrt{\frac{1-g(t)^2}{1-t^2}}\, v
    \end{equation}
    is of class $C^\alpha$. Since every $g\in H_\alpha$ is compactly supported in $(-1,1)$, there exists $\varepsilon>0$ such that $g(t)=t$ for every $|t|>1-\varepsilon$ and every $g\in H_{\alpha}$. Thus, $\hat{g}$ agrees with the identity on a neighbourhood of the poles (see Figure~\ref{Fig:maps-g-compactly-supported}). It therefore remains to verify that, for every $g\in H_{\alpha}$, the map \eqref{eq: projection 1st coordinate hat g} is of class $C^\alpha$ on the region $|t|<1-\varepsilon/2$. On this region the function $t\mapsto(1-g(t)^2)/(1-t^2)$ is of class $C^\alpha$, positive, and bounded away from zero. Therefore, its square root is also of class $C^\alpha$, and hence so is \eqref{eq: projection 1st coordinate hat g}. This finishes the proof of the lemma.
\end{proof}

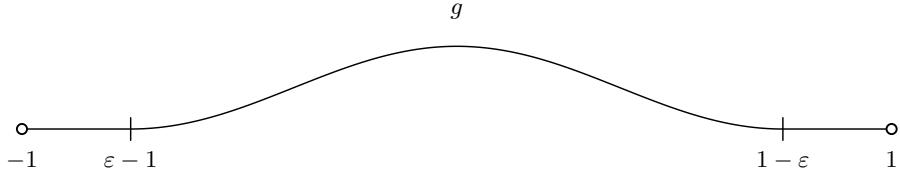
\begin{figure}[ht]
\centering
\begin{tikzpicture}[x=1.15cm,y=1.15cm,line cap=round,line join=round,
  endpoint/.style={
    circle,
    draw=black,
    fill=white,
    line width=.72pt,
    inner sep=1.35pt
  }
]
  \def\a{1.25}
  \def\b{8.75}
  \def\A{0.95}

  \draw[line width=0.55pt] (0,0) -- (\a,0);
  \draw[line width=0.55pt]
    plot[samples=180,domain=\a:\b]
      (\x,{\A*(sin(180*(\x-\a)/(\b-\a)))^2});
  \draw[line width=0.55pt] (\b,0) -- (10,0);

  \foreach \x in {\a,\b}
    \draw[line width=0.55pt] (\x,-0.13) -- (\x,0.13);

  \node[endpoint] at (0,0) {};
  \node[endpoint] at (10,0) {};

  \node[font=\small,below=5pt] at (0,0) {$-1$};
  \node[font=\small,below=5pt] at (\a,0) {$\varepsilon-1$};
  \node[font=\small,below=5pt] at (\b,0) {$1-\varepsilon$};
  \node[font=\small,below=5pt] at (10,0) {$1$};
  \node[font=\small,above=7pt] at (5,\A) {$g$};

\end{tikzpicture}
\caption{The maps $g\in H_\alpha$ compactly supported on $(-1,1)$.}
\label{Fig:maps-g-compactly-supported}
\end{figure}

We then define 
\begin{equation}\label{eq: Group of critical regularity}
G_{n,\alpha}:=\langle\Gamma, R_1,\ldots,R_{n-1},\hat{H}_\alpha\rangle\le \Diff_+^\alpha(\s^n). 
\end{equation}

\subsection{Groups on the real projective space}\label{Subsec:projective group of crit reg}
As in the sphere case, we construct the desired group by enlarging $\tilde\Gamma$. We first adjoin suitable rotations that generate a dense subgroup of $\tilde K$, and then adjoin a projective analogue of the Kim-Koberda group used above.

Define 
\begin{equation*}
 \tilde K:=\tilde\rho\left\lbrace \begin{pmatrix}
    B & \vec{0}^t\\
    \vec{0} & 1
\end{pmatrix} : B\in \SO(n) \right\rbrace. 
\end{equation*}
Although $\tilde\rho$ is not injective when $n$ is odd, its restriction to this subgroup is injective, and hence $\SO(n) \cong \tilde K$. Indeed, if a matrix $\operatorname{diag}(B, 1)$ acts trivially on $\RP^n$, then it is a scalar matrix, and since its last diagonal entry is $1$, we must have $B = I$. Likewise, let $\tilde R_1, \dots, \tilde R_{n-1}$ be the images under $\tilde\rho$ of the rotations $\rho^{-1}(R_1), \dots, \rho^{-1}(R_{n-1})$ introduced in Section~\ref{Subsec:sphere group of crit reg}. It follows from Lemma~\ref{Lem:exists-dense-rotations} that $\langle \tilde R_1, \dots, \tilde R_{n-1}\rangle $ is dense in $\tilde K$. 

We now construct the projective analogue of the subgroup $\hat{H}_\alpha$ used in the sphere case. The construction is similar to \eqref{eq: hat g}, but starts with a copy of $H_\alpha$ supported in $(0,1)$. We first extend these elements oddly to $[-1,1]$, and then apply the sphere construction to these odd extensions. The resulting diffeomorphisms commute with the antipodal map and hence descend to $\RP^n$.

By Theorem~\ref{thm: kim-koberda}, after an affine conjugation, we may fix a finitely generated compactly supported group $H_\alpha \leq \Diff_c^\alpha( (0,1))$ satisfying the conclusions of the theorem. Then every $g \in H_\alpha$ acts as the identity outside a uniform compact interval. For $g \in H_\alpha$, define its \textit{odd extension} $\E_g \colon [-1, 1] \to [-1, 1]$ by  
\begin{equation}\label{Eq:g-odd-definition}
\E_g(t) = \begin{cases} g(t), & t \geq 0, \\ -g(-t), & t \leq 0.\end{cases} 
\end{equation}
Since $g$ acts as the identity on a neighbourhood of $0$ (see Figure~\ref{Fig:odd-extension-interval}), the map $\E_g$ is a $C^\alpha$-diffeomorphism of $[-1, 1]$. It also agrees with the identity in a neighbourhood of the endpoints and satisfies $\E_g(-t)=-\E_g(t)$ for every $t \in [-1,1]$. Moreover, for $t \leq 0$ we have 
\[
\E_g(\E_h(t)) = -g(h(-t)) = \E_{gh}(t).
\] Thus, the assignment $g \mapsto \E_g$ is an injective homomorphism $H_\alpha \to \Diff^\alpha([-1,1])$. 

\begin{figure}[ht]
\centering
\begin{tikzpicture}[x=1.05cm,y=1.05cm,line cap=round,line join=round,
  endpoint/.style={
    circle,
    draw=black,
    fill=white,
    line width=.72pt,
    inner sep=1.35pt
  }
]
  \def\xmin{0}
  \def\leftinner{1.10}
  \def\minusE{4.35}
  \def\zero{5.00}
  \def\plusE{5.65}
  \def\rightinner{8.90}
  \def\xmax{10.00}
  \def\A{0.85}

  \def\flatness{0.1}

  \draw[line width=0.55pt] (\xmin,0) -- (\leftinner,0);

  \draw[line width=0.55pt]
    plot[samples=160,domain=\leftinner:\minusE]
      (\x,{-\A*(sin(180*(\x-\leftinner)/(\minusE-\leftinner)))^2
                *(1+\flatness*(cos(180*(\x-\leftinner)/(\minusE-\leftinner)))^2)});

  \draw[line width=0.55pt] (\minusE,0) -- (\plusE,0);

  \draw[line width=0.55pt]
    plot[samples=160,domain=\plusE:\rightinner]
      (\x,{\A*(sin(180*(\x-\plusE)/(\rightinner-\plusE)))^2
               *(1+\flatness*(cos(180*(\x-\plusE)/(\rightinner-\plusE)))^2)});

  \draw[line width=0.55pt] (\rightinner,0) -- (\xmax,0);

  \foreach \x in {\leftinner,\minusE,\zero,\plusE,\rightinner}
    \draw[line width=0.55pt] (\x,-0.13) -- (\x,0.13);

  \node[endpoint] at (\xmin,0) {};
  \node[endpoint] at (\xmax,0) {};

  \node[font=\small,below=5pt] at (\xmin,0) {$-1$};
  \node[font=\small,below=5pt] at (\leftinner,0) {$\varepsilon-1$};
  \node[font=\small,above=5pt] at (\minusE,0) {$-\varepsilon$};
  \node[font=\small,below=5pt] at (\zero,0) {$0$};
  \node[font=\small,above=5pt] at (\plusE,0) {$\varepsilon$};
  \node[font=\small,below=5pt] at (\rightinner,0) {$1-\varepsilon$};
  \node[font=\small,below=5pt] at (\xmax,0) {$1$};
  \node[font=\small,above=7pt] at ({(\plusE+\rightinner)/2},\A) {$\E_g$};
\end{tikzpicture}
\caption{The odd extension $\E_g$ of the maps $g\in H_{\alpha}\leq \Diff_{+}^{\alpha}(0,1)$, supported away from $\{-1,0,1\}$.}
\label{Fig:odd-extension-interval}
\end{figure}
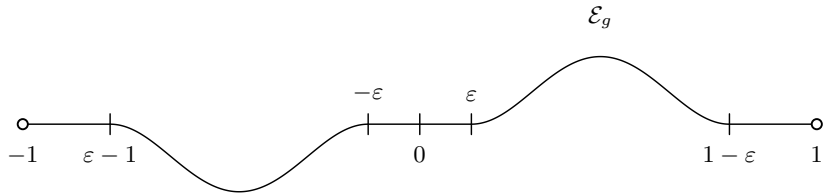

We may therefore apply the construction \eqref{eq: hat g} with $\E_g$ in place of $g$. Together with Lemma~\ref{lem: hat H_alpha}, this gives an embedding 
\[
H_\alpha\to\Diff^\alpha(\s^n),\quad g\mapsto\hat{g},
\]
where
\begin{equation}\label{eq: hat g projective}
\hat{g}(v,t) :=\left(\sqrt{\frac{1-\E_g(t)^2}{1-t^2}}\, v,\E_g(t)\right),
\end{equation}
away from the poles, and the map is extended to the poles as before.  

Now, since $\E_g$ is odd we have $\hat{g}(-v,-t)=-\hat{g}(v,t)$. Hence, $\hat{g}$ commutes with the antipodal map and therefore descends to a $C^\alpha$-diffeomorphism $\tilde{g} \colon \RP^n \to \RP^n$ satisfying $\pi \circ \hat{g} = \tilde{g} \circ \pi.$


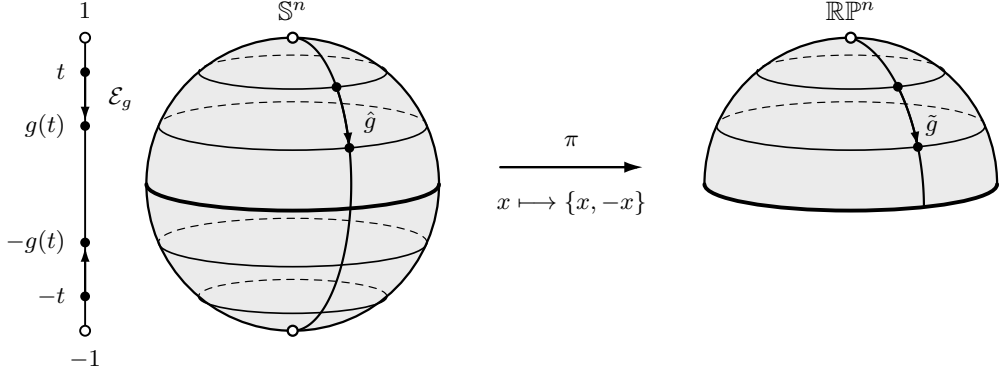
\begin{figure}[h]
\centering

\begin{tikzpicture}[scale=0.9,
  >=Latex,
  line cap=round,
  line join=round,
  font=\small,
  outline/.style={line width=.85pt},
  sphere fill/.style={fill=black!8},
  equator front/.style={line width=1.45pt},
  latitude front/.style={line width=.62pt},
  latitude back/.style={
    line width=.45pt,
    dashed,
    dash pattern=on 2.8pt off 2.4pt
  },
  meridian/.style={line width=.82pt},
  hidden meridian/.style={
    line width=.55pt,
    dashed,
    dash pattern=on 3pt off 2.5pt
  },
  action/.style={
    -{Latex[length=2mm,width=1.2mm]},
    line width=.82pt
  },
  projection/.style={
    -{Latex[length=2.5mm,width=1.5mm]},
    line width=.9pt
  },
  interval/.style={line width=.72pt},
  point/.style={
    circle,
    fill=black,
    inner sep=1.3pt
  }
]


\def\R{2.15}
\def\flat{0.18}
\def\M{0.86}

\pgfmathsetmacro{\be}{\flat*\R}

\def\angA{58}
\def\angB{32}

\pgfmathsetmacro{\xE}{\R*cos(300)}
\pgfmathsetmacro{\yE}{\be*sin(300)}

\pgfmathsetmacro{\yt}{\yE*cos(\angA) + \R*sin(\angA)}
\pgfmathsetmacro{\yg}{\yE*cos(\angB) + \R*sin(\angB)}

\pgfmathsetmacro{\at}{sqrt(\R*\R - \yt*\yt)}
\pgfmathsetmacro{\bt}{\flat*\at}

\pgfmathsetmacro{\ag}{sqrt(\R*\R - \yg*\yg)}
\pgfmathsetmacro{\bg}{\flat*\ag}

%

\def\angSt{41.52286}
\def\angSg{14.36054}


\def\angPt{49.94632}
\def\angPg{23.43582}

\pgfmathsetmacro{\angmidS}{0.5*(\angSt+\angSg)}
\pgfmathsetmacro{\angmidP}{0.5*(\angPt+\angPg)}

\def\Ix{-3.05}
\def\Ih{2.15}


\begin{scope}[xshift=-4.1cm]


  \draw[interval] (\Ix,-\Ih) -- (\Ix,\Ih);

  \draw[interval,fill=white]
    (\Ix,\Ih) circle[radius=.073];

  \draw[interval,fill=white]
    (\Ix,-\Ih) circle[radius=.073];

  \node[above=4pt] at (\Ix,\Ih) {$1$};
  \node[below=4pt] at (\Ix,-\Ih) {$-1$};

  \node[point] at (\Ix,\yt) {};
  \node[point] at (\Ix,\yg) {};
  \node[point] at (\Ix,-\yg) {};
  \node[point] at (\Ix,-\yt) {};

  \node[anchor=east] at (\Ix-.15,\yt) {$t$};
  \node[anchor=east] at (\Ix-.15,\yg) {$g(t)$};
  \node[anchor=east] at (\Ix-.15,-\yg) {$-g(t)$};
  \node[anchor=east] at (\Ix-.15,-\yt) {$-t$};

  \draw[action]
    (\Ix,\yt-.06) --
    (\Ix,\yg+.06)
    node[midway,right=5pt] {$\E_g$};

  \draw[action]
    (\Ix,-\yt+.06) --
    (\Ix,-\yg-.06);


  \fill[sphere fill]
    (0,0) circle[radius=\R];

  \draw[latitude back]
    (0,\yt) ++(\at,0)
    arc[
      start angle=0,
      end angle=180,
      x radius=\at,
      y radius=\bt
    ];

  \draw[latitude back]
    (0,\yg) ++(\ag,0)
    arc[
      start angle=0,
      end angle=180,
      x radius=\ag,
      y radius=\bg
    ];

  \draw[latitude back]
    (0,-\yg) ++(\ag,0)
    arc[
      start angle=0,
      end angle=180,
      x radius=\ag,
      y radius=\bg
    ];

  \draw[latitude back]
    (0,-\yt) ++(\at,0)
    arc[
      start angle=0,
      end angle=180,
      x radius=\at,
      y radius=\bt
    ];


  \begin{scope}
    \clip (0,0) circle[radius=\R];

    \draw[meridian]
      plot[
        domain=-90:90,
        samples=180
      ]
      ({\M*cos(\x)},{\R*sin(\x)});

    \draw[action]
      plot[
        domain=\angSt:\angSg,
        samples=60
      ]
      ({\M*cos(\x)},{\R*sin(\x)});
  \end{scope}


  \draw[latitude front]
    (0,\yt) ++(-\at,0)
    arc[
      start angle=180,
      end angle=360,
      x radius=\at,
      y radius=\bt
    ];

  \draw[latitude front]
    (0,\yg) ++(-\ag,0)
    arc[
      start angle=180,
      end angle=360,
      x radius=\ag,
      y radius=\bg
    ];

  \draw[latitude front]
    (0,-\yg) ++(-\ag,0)
    arc[
      start angle=180,
      end angle=360,
      x radius=\ag,
      y radius=\bg
    ];

  \draw[latitude front]
    (0,-\yt) ++(-\at,0)
    arc[
      start angle=180,
      end angle=360,
      x radius=\at,
      y radius=\bt
    ];

  \draw[equator front]
    (0,0) ++(-\R,0)
    arc[
      start angle=180,
      end angle=360,
      x radius=\R,
      y radius=\be
    ];

  \draw[outline]
    (0,0) circle[radius=\R];

  \draw[outline,fill=white]
    (0,\R) circle[radius=.073];

  \draw[outline,fill=white]
    (0,-\R) circle[radius=.073];


  \coordinate (Pt) at
    ({\M*cos(\angSt)},
     {\R*sin(\angSt)});

  \coordinate (Pg) at
    ({\M*cos(\angSg)},
     {\R*sin(\angSg)});

  \node[point] at (Pt) {};
  \node[point] at (Pg) {};

  \node[anchor=west]
    at ({\M*cos(\angmidS)+.15},
        {\R*sin(\angmidS)-0.1})
    {$\hat g$};

  \node[font=\bfseries]
    at (0,2.55)
    {$\Sn$};

\end{scope}


\draw[projection]
  (-1.05,.25) -- (1.05,.25)
  node[midway,above=4pt] {$\pi$}
  node[midway,below=5pt]
  {$x\longmapsto \{x,-x\}$};


\begin{scope}[xshift=4.1cm]

  \path[sphere fill]
    (-\R,0)
    arc[
      start angle=180,
      end angle=0,
      radius=\R
    ]
    arc[
      start angle=0,
      end angle=-180,
      x radius=\R,
      y radius=\be
    ]
    -- cycle;


  \draw[latitude back]
    (0,\yt) ++(\at,0)
    arc[
      start angle=0,
      end angle=180,
      x radius=\at,
      y radius=\bt
    ];

  \draw[latitude back]
    (0,\yg) ++(\ag,0)
    arc[
      start angle=0,
      end angle=180,
      x radius=\ag,
      y radius=\bg
    ];

  \draw[meridian,line cap=butt]
    plot[
      domain=90:0,
      samples=140
    ]
    ({\xE*cos(\x)},
     {\yE*cos(\x)+\R*sin(\x)});

  \draw[latitude front]
    (0,\yt) ++(-\at,0)
    arc[
      start angle=180,
      end angle=360,
      x radius=\at,
      y radius=\bt
    ];

  \draw[latitude front]
    (0,\yg) ++(-\ag,0)
    arc[
      start angle=180,
      end angle=360,
      x radius=\ag,
      y radius=\bg
    ];


  \coordinate (Q1) at
    ({\xE*cos(\angPt)},
     {\yE*cos(\angPt)+\R*sin(\angPt)});

  \coordinate (Q2) at
    ({\xE*cos(\angPg)},
     {\yE*cos(\angPg)+\R*sin(\angPg)});

  \node[point] at (Q1) {};
  \node[point] at (Q2) {};

  \draw[action]
    plot[
      domain=\angPt:\angPg,
      samples=60
    ]
    ({\xE*cos(\x)},
     {\yE*cos(\x)+\R*sin(\x)});

  \node[anchor=west]
    at ({\xE*cos(\angmidP)+.10},
        {\yE*cos(\angmidP)+\R*sin(\angmidP)-0.15})
    {$\tilde g$};

  \draw[outline]
    (-\R,0)
    arc[
      start angle=180,
      end angle=0,
      radius=\R
    ];

  \draw[equator front]
    (0,0) ++(-\R,0)
    arc[
      start angle=180,
      end angle=360,
      x radius=\R,
      y radius=\be
    ];

  \draw[outline,fill=white]
    (0,\R) circle[radius=.073];

  \node[font=\bfseries]
    at (0,2.55)
    {$\RP^n$};

\end{scope}

\end{tikzpicture}

\caption{
The odd extension $\E_g$ gives rise to the sphere
diffeomorphism $\hat g$, which descends through
$\pi:\Sn\to\mathbb{RP}^n$ to the induced diffeomorphism
$\tilde g$.
}
\label{fig:action-rpn}
\end{figure}

\begin{lemma}\label{Lem:gproj-injects-indo-diffalpha}
The map $g \mapsto \tilde{g}$ defines an injective homomorphism $H_\alpha \rightarrow \Diff^\alpha(\RP^n)$. 
\end{lemma}
\begin{proof}
The homomorphism property follows from the corresponding property upstairs in $\Sn$. We only need to prove injectivity. Suppose that $\tilde {g} = \id_{\RP^n}$. Then since $\pi \circ \hat{g} = \tilde{g} \circ \pi$ we have $\pi \circ \hat{g} = \pi$. So, for every $x\in\s^n$ we have $\hat{g}(x)\in \{-x,x\}$. By continuity and connectedness of $\s^n$, it follows that $\hat{g} \in \{ \id_{\Sn}, A\}$. But $\E_g$ is the identity on a neighbourhood of $0$, so \eqref{eq: hat g projective} shows that $\hat{g}$ is the identity on a neighbourhood of the equator. Hence $\hat{g}  \neq A$, so that $\hat{g} = \id_{\Sn}$. Using the final coordinate in \eqref{eq: hat g projective}, this shows that $\E_g(t) = t$ for every $t \in [-1,1]$, so that $g = \id_{(0,1)}$. Thus the map is injective. 
\end{proof}

Let $\tilde H_\alpha$ denote the image of $H_\alpha$ under the embedding of Lemma~\ref{Lem:gproj-injects-indo-diffalpha}, and define
\begin{equation}\label{Eq:G-proj-def}
\tilde G_{n,\alpha} := \big\langle \tilde \Gamma, \tilde R_1, \dots, \tilde R_{n-1}, \tilde{H}_\alpha \big\rangle \leq \Diff^\alpha(\RP^n). 
\end{equation}
Since $\tilde \Gamma$ and $\tilde H_\alpha$ are finitely generated, so is $\tilde G_{n,\alpha}$. 

\section{Proof of Theorem \ref{thm: main}}

\noindent We now prove our main result. The arguments for $\s^n$ and $\RP^n$ follow the same general strategy. In both cases, after using rigidity to normalize a hypothetical $C^\beta$-embedding, we show that the image of the corresponding Kim-Koberda subgroup centralizes the appropriate rotation group. We then pass to the quotient by this rotation group and obtain an action on the interval. 

 For $n \geq 2$, we regard $\Sn \subseteq \R^n \times \R$, and define the latitude map $\lambda \colon \Sn \to [-\pi/2, \pi/2]$ by $\lambda(v,t)=\arcsin(t)$. For $s \in (-\pi/2, \pi/2)$, every point $(v,t) \in \Sn$ with $\lambda(v,t) = s$ can be written uniquely as $(u\cos(s), \sin(s))$ where $u \in \mathbb{S}^{n-1}$. Since the subgroup $K\le\Diff_+^\infty(\s^n)$ defined in \eqref{eq: def of subg K isom SO(n)} preserves the last coordinate and acts transitively on each latitude, the fibers of $\lambda$ are precisely the $K$-orbits. Thus, $\lambda$ induces a continuous bijection $\overline{\lambda}:\Sn / K \to [-\pi/2, \pi/2]$. Since $\Sn / K$ is compact and the interval is Hausdorff, $\overline{\lambda}$ is a homeomorphism. Therefore, 
 \begin{equation}\label{eq: interval quotient Sn}
 \frac{\s^n}{K}\simeq \left[-\frac{\pi}{2},\frac{\pi}{2}\right].
 \end{equation}
The latitude map satisfies $\lambda(-v,-t)=-\lambda(v,t)$, and hence $|\lambda|$ is invariant under the antipodal map. It therefore descends to a continuous map $\tilde\lambda:\RP^n\to[0,\pi/2 ]$ satisfying $\tilde\lambda\circ\pi=|\lambda|$. The fibers of $\tilde \lambda$ are precisely the $\tilde K$-orbits. Indeed, $\tilde K$ preserves $|t|$, while $K\cong\SO(n)$ acts on each sphere transitively on the first $n$ coordinates. Hence $\tilde \lambda$ induces a homeomorphism
\begin{equation}\label{eq: interval quotient RPn}
    \frac{\RP^n}{\tilde K}\simeq \left[0,\frac{\pi}{2}\right].
\end{equation}

The following proposition provides the quotient maps that will be used in both cases (see Figure~\ref{fig:latitudes}). Its proof is postponed to Section~\ref{Sec:proof map Pi_n}. 
\begin{proposition}\label{prop: map Pi_n}
Let $n\ge2$ and $\beta>1$. Under the identifications \eqref{eq: interval quotient Sn} and \eqref{eq: interval quotient RPn}, consider the assignments
\begin{equation}\label{eq: assignment psi_f}
f\mapsto\psi_f,\quad \psi_f(K\cdot v):=K\cdot f(v),\quad \text{for}\quad f\in \Centr_{\Diff^\beta(\s^n)}(K),
\end{equation} and \begin{equation}\label{eq: assignment tilde psi_f}
f\mapsto\tilde\psi_f,\quad \tilde\psi_f(\tilde K\cdot \pi(v)):=\tilde K\cdot f(\pi(v)),\quad \text{for}\quad f\in \Centr_{\Diff^\beta(\RP^n)}(\tilde K).
\end{equation}
Then these assignments define group homomorphisms 
\[
\Pi_n:\Centr_{\Diff^\beta(\s^n)}(K)\to \Diff^\beta\left(\left[-\pi/2,\pi/2\right]\right),
\] and \[
\tilde \Pi_n:\Centr_{\Diff^\beta(\RP^n)}(\tilde K)\to \Diff^\beta\left(\left[0,\pi/2\right]\right).
\]
Moreover, both kernels are abelian. More precisely $\ker(\Pi_2)$ and $\ker(\tilde\Pi_2)$ are abelian, and for every $n\ge 3$, $\ker(\Pi_n)\cong\ker(\tilde\Pi_n)\cong\Z/2\Z$.
\end{proposition}

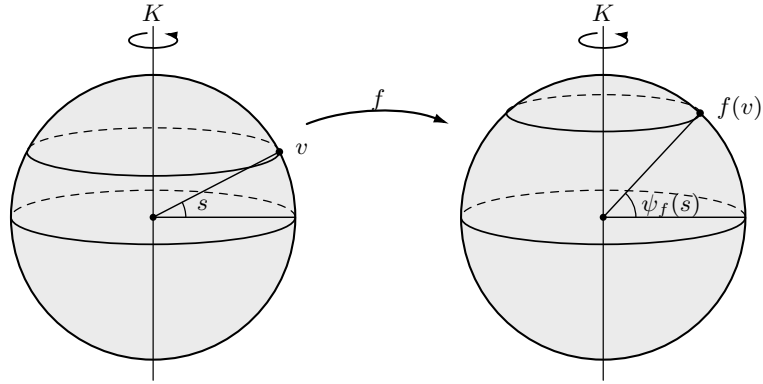
\begin{figure}[h]
\centering
\begin{tikzpicture}[scale=0.8,
  line cap=round,
  line join=round,
  >=Latex,
  font=\small,
  sphere fill/.style={fill=black!8},
  sphere outline/.style={line width=.78pt},
  latitude front/.style={line width=.65pt},
  latitude back/.style={line width=.48pt,dashed,dash pattern=on 3pt off 2.6pt},
  axis line/.style={line width=.46pt},
  radius line/.style={line width=.54pt},
  angle mark/.style={line width=.48pt},
  action arrow/.style={-{Latex[length=2.1mm,width=1.25mm]},line width=.82pt},
  rotation arrow/.style={-{Latex[length=1.85mm,width=1.1mm]},line width=.68pt}
]

\def\R{2.35}          
\def\flat{0.19}       
\def\sep{3.72}        
\def\yL{1.08}         
\def\yR{1.72}         
\def\axisbelow{0.34}  
\def\axisabove{0.80}  
\def\rotY{0.57}       
\def\rotXrad{0.40}    
\def\rotYrad{0.125}   

\pgfmathsetmacro{\aL}{sqrt(\R*\R-\yL*\yL)}
\pgfmathsetmacro{\bL}{\flat*\aL}
\pgfmathsetmacro{\aR}{sqrt(\R*\R-\yR*\yR)}
\pgfmathsetmacro{\bR}{\flat*\aR}
\pgfmathsetmacro{\be}{\flat*\R}
\pgfmathsetmacro{\thetaL}{asin(\yL/\R)}
\pgfmathsetmacro{\thetaR}{asin(\yR/\R)}
\pgfmathsetmacro{\halfThetaL}{0.5*\thetaL}
\pgfmathsetmacro{\halfThetaR}{0.5*\thetaR}

\newcommand{\Krotation}{%
  \draw[rotation arrow]
    (0,{\R+\rotY}) ++(105:{\rotXrad} and {\rotYrad})
    arc[start angle=105,end angle=425,
        x radius=\rotXrad,y radius=\rotYrad];
  \node at (0,{\R+1.02}) {$K$};
}

\begin{scope}[shift={(-\sep,0)}]
  \fill[sphere fill] (0,0) circle[radius=\R];
  \draw[sphere outline] (0,0) circle[radius=\R];

  \draw[latitude back] (0,0) ++(0:\R)
    arc[start angle=0,end angle=180,x radius=\R,y radius=\be];
  \draw[latitude front] (0,0) ++(180:\R)
    arc[start angle=180,end angle=360,x radius=\R,y radius=\be];

  \draw[latitude back] (0,\yL) ++(0:\aL)
    arc[start angle=0,end angle=180,x radius=\aL,y radius=\bL];
  \draw[latitude front] (0,\yL) ++(180:\aL)
    arc[start angle=180,end angle=360,x radius=\aL,y radius=\bL];

  \draw[axis line] (0,{-\R-\axisbelow}) -- (0,{\R+\axisabove});

  \coordinate (CL) at (0,0);
  \coordinate (v) at (\aL,\yL);
  \draw[radius line] (CL) -- (\R,0);
  \draw[radius line] (CL) -- (v);

  \draw[angle mark] (0.54,0)
    arc[start angle=0,end angle=\thetaL,radius=.54];
  \node at ({0.84*cos(\halfThetaL)},{0.84*sin(\halfThetaL)}) {$s$};

  \fill (CL) circle[radius=.058];
  \fill (v) circle[radius=.062];
  \node[anchor=west] at ($(v)+(0.11,0.05)$) {$v$};

  \Krotation
\end{scope}

\begin{scope}[shift={(\sep,0)}]
  \fill[sphere fill] (0,0) circle[radius=\R];
  \draw[sphere outline] (0,0) circle[radius=\R];

  \draw[latitude back] (0,0) ++(0:\R)
    arc[start angle=0,end angle=180,x radius=\R,y radius=\be];
  \draw[latitude front] (0,0) ++(180:\R)
    arc[start angle=180,end angle=360,x radius=\R,y radius=\be];

  \draw[latitude back] (0,\yR) ++(0:\aR)
    arc[start angle=0,end angle=180,x radius=\aR,y radius=\bR];
  \draw[latitude front] (0,\yR) ++(180:\aR)
    arc[start angle=180,end angle=360,x radius=\aR,y radius=\bR];

  \draw[axis line] (0,{-\R-\axisbelow}) -- (0,{\R+\axisabove});

  \coordinate (CR) at (0,0);
  \coordinate (fv) at (\aR,\yR);
  \draw[radius line] (CR) -- (\R,0);
  \draw[radius line] (CR) -- (fv);

  \draw[angle mark] (0.54,0)
    arc[start angle=0,end angle=\thetaR,radius=.54];
  \node[anchor=west] at (0.48,0.2) {$\psi_f(s)$};

  \fill (CR) circle[radius=.058];
  \fill (fv) circle[radius=.062];
  \node[anchor=west] at ($(fv)+(0.12,0.08)$) {$f(v)$};

  \Krotation
\end{scope}

\draw[action arrow]
  (-1.18,1.54)
  .. controls (-0.48,1.82) and (0.48,1.82) ..
  (1.18,1.54);
\node at (0,1.94) {$f$};

\end{tikzpicture}

\caption{Each $K$-orbit is a latitude on $S^n$. Since $f$ commutes with $K$, it preserves latitudes, and hence descends to a map $\psi_f$ on $[-\pi/2,\pi/2]$.}
\label{fig:latitudes}
\end{figure}

We next establish the corresponding centralizer inclusions for the normalized embeddings. These are the key steps that allow us to apply the preceding proposition.
\begin{proposition}\label{Prop:hatQ-embeds-in-centralizer}
    Let $n\ge2$, let $\beta>\alpha\ge1$, and let 
    \[
    \sigma:G\to \Diff^\beta(M)
    \]
    be an embedding. Then:
    \begin{enumerate}
    \item\label{item: 1 proposition image varrho hat H} if $(G,M)=(G_{n,\alpha},\s^n)$ and $\sigma |_\Gamma=\id_\Gamma$ we have $\sigma(\hat{H}_\alpha) \leq \Centr_{\Diff^\beta(\s^n)}(K)$,
    \item\label{item: 2 proposition image varrho tilde H} if $(G,M)=(\tilde G_{n,\alpha},\RP^n)$ and $\sigma |_{\tilde\Gamma}=\id_{\tilde\Gamma}$ we have $\sigma(\tilde{H}_\alpha) \leq \Centr_{\Diff^\beta(\RP^n)}(\tilde K)$.
    \end{enumerate}
\end{proposition}
\begin{proof}
    We first prove \eqref{item: 1 proposition image varrho hat H}. Let $g \in \hat{H}_\alpha$. By construction $[g, R_i] = 1$ for every $1 \leq i \leq n-1$. Therefore also $[\sigma(g), \sigma(R_i)] = 1$, so that $[\sigma(g), \sigma(R_i)^2] = 1$. Apply Proposition~\ref{Prop:Recognition-prop} to $\sigma |_{\langle \Gamma, R_i\rangle}$. Since $\sigma |_\Gamma = \id_\Gamma$, by assumption, we find $\sigma(R_i)^2 = R_i^2$. Hence $[\sigma(g), R_i^2] = 1$ for all $1 \leq i \leq n-1$. Thus $\sigma(g)$ centralizes the subgroup $\langle R_1^2, \dots, R_{n-1}^2\rangle \leq K$. But this subgroup is dense in $K$ by Lemma~\ref{Lem:exists-dense-rotations}. Since the centralizer of a fixed homeomorphism is closed, $\sigma(g)$ centralizes its closure, and hence all of $K$. Thus $\sigma(g) \in \Centr_{\Diff^\beta(\s^n)}(K)$, and since $g$ was arbitrary we are done. 

    We now prove \eqref{item: 2 proposition image varrho tilde H}. Let $\tilde{g} \in \tilde{H}_\alpha$. Let $\hat g$ be the map defined by \eqref{eq: hat g projective}, so that $\pi\circ\hat g=\tilde g\circ\pi$. By its definition, $\hat g$ commutes with every element of $K$. Passing through to the quotient map $\pi \colon \Sn \to \RP^n$, we conclude that $\tilde{g}$ commutes with every element of $\tilde K$, and hence in particular $[\tilde{g}, \tilde R_i] = 1$ for all $1 \leq i \leq n-1$. Applying $\sigma$ to these relations, the remainder of the argument proceeds exactly as in the previous case, using the projective map recognition property from Proposition~\ref{Prop:Recognition-prop} and the density of $\langle\tilde R_1,\ldots,\tilde R_{n-1} \rangle$ in $\tilde K$.
\end{proof}
With these ingredients in hand, we turn to the proof of Theorem~\ref{thm: main}. 
\subsection{The sphere case}\label{Sec:proof-of-sphere}
We now prove that the group $G_{n,\alpha}$, defined in \eqref{eq: Group of critical regularity}, has critical regularity equal to $\alpha$ on $\s^n$, thereby proving Theorem~\ref{thm: main} in this case. Since $G_{n,\alpha}\le\Diff^\alpha(\s^n)$, it suffices to show that $G_{n,\alpha}$ does not embed into $\Diff^\beta(\s^n)$ for any $\beta>\alpha$. Suppose, for contradiction, that there exists an injective homomorphism 
\[
\sigma:G_{n,\alpha}\to \Diff^\beta(\s^n),\quad\text{for some }\beta>\alpha.
\]
By the $C^\beta$-rigidity of $\Gamma$, we may assume that the restriction of $\sigma$ to $\Gamma$ is the identity. Proposition~\ref{Prop:hatQ-embeds-in-centralizer} then gives $\sigma(\hat{H}_\alpha) \leq \Centr_{\Diff^\beta(\s^n)}(K)$. Let $\Pi_n$ be the homomorphism given in Proposition~\ref{prop: map Pi_n}, and consider the restriction to $[H_\alpha,H_\alpha]$ of the homomorphism
\[
H_\alpha\overset{g\mapsto\hat{g}}{\xrightarrow{\hspace{0,7cm}}}\hat{H}_\alpha\overset{\Pi_n\circ\sigma}{\xrightarrow{\hspace{0.7cm}}}\Diff^\beta([-\pi/2,\pi/2]).
\] 
Since $[H_\alpha,H_\alpha]$ is simple, the kernel of this restriction is either trivial or the whole group. The latter is impossible: it would imply that $\sigma([\hat{H}_\alpha,\hat{H}_\alpha])\subset \ker(\Pi_n),$ whereas $\ker(\Pi_n)$ is abelian and $\sigma([\hat{H}_\alpha,\hat{H}_\alpha])$ is nonabelian. Therefore, $[H_\alpha,H_\alpha]$ embeds into $\Diff^\beta([-\pi/2,\pi/2])$. Composing this embedding with the orientation sign homomorphism gives a homomorphism from $[H_\alpha,H_\alpha]$ to $\{\pm1\}$. Since $[H_\alpha,H_\alpha]$ is nonabelian and simple, this homomorphism must be trivial. Hence, $[H_\alpha,H_\alpha]$ embeds into $\Diff_+^\beta([-\pi/2,\pi/2])$ contradicting Theorem~\ref{thm: kim-koberda}. 
\subsection{The projective space case} In Section~\ref{Subsec:projective group of crit reg} we constructed finitely generated groups $\tilde G_{n,\alpha} \leq \Diff^\alpha(\RP^n)$. We now show that these groups do not embed into $\Diff^\beta(\RP^n)$ for any $\beta > \alpha$. Suppose that 
\[
\sigma \colon \tilde G_{n, \alpha} \hookrightarrow \Diff^\beta(\RP^n)\quad\text{for some}\quad\beta > \alpha.
\]
Since $\beta > \alpha \geq 1$, in particular $\beta > 1$, and hence we may apply Proposition~\ref{Prop:Recognition-prop}. Thus, after conjugating $\sigma$ by a $C^\beta$-diffeomorphism of $\RP^n$, we may assume that $\sigma \mid_{\tilde\Gamma} = \id_{\tilde\Gamma}$. Proposition~\ref{Prop:hatQ-embeds-in-centralizer} then gives $\sigma(\tilde{H}_\alpha) \leq \Centr_{\Diff^\beta(\RP^n)}(\tilde K)$. As in the sphere case, consider the homomorphism 
\[
H_\alpha\overset{g\mapsto\tilde{g}}{\xrightarrow{\hspace{0,7cm}}}\tilde{H}_\alpha\overset{\tilde\Pi_n\circ\sigma}{\xrightarrow{\hspace{0.7cm}}}\Diff^\beta([0,\pi/2]).
\]
Restricting the homomorphism to $[H_\alpha,H_\alpha]$, the remainder of the argument is exactly the same as in the sphere case: the simplicity of $[H_\alpha,H_\alpha]$, together with the fact that $\ker(\tilde\Pi_n)$ is abelian, implies that this restriction is injective, yielding the desired contradiction.

\section{Proof of Proposition \ref{prop: map Pi_n}}\label{Sec:proof map Pi_n}
\noindent We begin by establishing the following technical lemma used in the proof of Proposition~\ref{prop: map Pi_n}. 
\begin{lemma}\label{Lem:endpoint-regularity}
Let $\beta \geq 1$. Suppose that $u \colon (-\varepsilon, \varepsilon) \to \R^d$ is $C^\beta$, $u(0) = 0$, and $u'(0) \neq 0$. Then  
\[
\nu(t) = \begin{cases}
\lVert u(t) \rVert, & t \geq 0,\\
-\lVert u(t) \rVert, & t < 0 ,\\
\end{cases}
\]
is $C^\beta$ near $0$ and $\nu'(0) = \lVert u'(0) \rVert>0$. 
\end{lemma}
\begin{proof}
Write $u(t) = ta(t)$, where $a(t) = \int_0^1 u'(st) ds$. Then $a$ is $C^{\beta-1}$, since $u'(t)$ is $C^{\beta-1}$, and $a$ is non-zero near $0$, since $a(0)=u'(0) \neq 0$ by assumption. For all $t$ near $0$, we have $\nu(t) = t \lVert a(t) \lVert$ so that $\nu$ is $C^{\beta-1}$. If we set $w(t) = a(t)/\lVert a(t) \rVert$, then for $t \neq 0$ we have $u(t)/\nu(t) = (ta(t))/(t\lVert a(t) \rVert) = w(t)$. For $t>0$ resp.\ $t<0$ we have $\nu(t) = \lVert u(t) \rVert$ resp.\ $\nu(t) = -\lVert u(t) \rVert$, so
\[
\nu'(t) = \left\langle \frac{u(t)}{\nu(t)}, u'(t)\right\rangle = \langle w(t), u'(t) \rangle \quad (t \neq 0)
\]
where $\langle \cdot, \cdot \rangle$ denotes the Euclidean inner product. 

Now set $g(t) := \langle w(t), u'(t)\rangle$ for all $t$. Since $w$ and $u'$ are both $C^{\beta-1}$, so too is $g$. Moreover, $g(0) = \langle u'(0)/\lVert u'(0) \rVert, u'(0)\rangle = \lVert u'(0) \rVert$. Thus $\nu'(t) = g(t)$ for $t \neq 0$. It remains only to check that $\nu'(0) = g(0)$. But from $\nu(t) = t \lVert a(t) \rVert$ and $\nu(0) = 0$, we have 
\[
\nu'(0) = \lim_{t \to 0} \frac{\nu(t) - \nu(0)}{t} = \lim_{t \to 0} \lVert a(t) \rVert = \lVert a(0) \rVert = \lVert u'(0) \rVert = g(0).
\]
Hence $\nu' = g$ on a neighbourhood of $0$. Since $g$ is $C^{\beta-1}$, it follows that $\nu$ is $C^\beta$. Finally, $\nu'(0) = \lVert u'(0) \rVert > 0$, so by the inverse function theorem $\nu$ is a local $C^\beta$-diffeomorphism at $0$. 
\end{proof}

\begin{proof}[Proof of Proposition \ref{prop: map Pi_n}]
We first prove that the assignment \eqref{eq: assignment psi_f} gives the desired homomorphism $\Pi_n$. Suppose that $f \in \Centr_{\Diff^0(\Sn)}(K)$. Then the universal property of quotient spaces gives a unique continuous $\psi_f \colon [-\pi/2, \pi/2] \to [-\pi/2, \pi/2]$ such that $\lambda \circ f = \psi_f \circ \lambda$. Applying this to $f^{-1}$ gives $\psi_f^{-1} = \psi_{f^{-1}}$, so that $\psi_f$ is a homeomorphism, and uniqueness gives $\psi_{fg} = \psi_f\circ \psi_g$. Hence we have a homomorphism into $\Diff^0\left([-\pi/2, \pi/2]\right)$. We will now show that if $f$ is $C^\beta$, then $\psi_f$ is also $C^\beta$. First, we prove regularity on the open interval $(-\pi/2, \pi/2)$. Note that the two poles are either fixed or interchanged. Define $\gamma(s) = (\cos(s) e_1, \sin(s))$ for $s \in (-\pi/2, \pi/2)$, where $e_1 \in \mathbb{S}^{n-1}$ is the unit vector $(1,0, \dots,0)$. Then $\lambda(\gamma(s)) = s$, so $\psi_f(s) = \lambda(f(\gamma(s)))$ by the universal property. Note that $f(\gamma(s))$ is also non-polar for non-polar $s$, since $\psi_f$ is a homeomorphism of the interval. Hence $\psi_f = \lambda(f(\gamma(s)))$ is $C^\beta$ away from the poles.

It remains to prove regularity at the endpoints, i.e.\ the two poles. Since $\psi_f$ is a homeomorphism of the interval $[-\pi/2,\pi/2]$, there is $\varepsilon \in \{ -1, 1\}$ such that $\psi_f\left(\pi/2\right) = \varepsilon \pi/2$ and $\psi_f\left(-\pi/2\right) = -\varepsilon \pi/2$. Consider first the north pole $N = (0,1)$ and let $\eta(t) = (\sin(t) e_1, \cos(t))$ for $t$ near $0$. For $t \geq 0$, we have $\eta(t) = \gamma\left( \pi/2 - t\right)$. Set $u(t) = \operatorname{pr}_{\R^n}( f(\eta(t)))$. Then $u(0) = 0$, since $f(N)$ is a pole (either $S$ or $N$). Furthermore, since $f$ is a $C^\beta$-diffeomorphism, $(Df)_N$ is an isomorphism. As $\eta'(0) = (e_1, 0) \neq 0$, we have $(Df)_N(\eta'(0)) \neq 0$. Moreover, the tangent space $T_{f(N)}(\Sn)$ is $\R^n \times \{ 0 \}$, so projection onto the first factor is injective there. Hence $u'(0) \neq 0$. Thus we can apply Lemma~\ref{Lem:endpoint-regularity}, and find that the signed norm
\[
\nu(t) = \begin{cases}
\lVert u(t) \rVert, & t \geq 0,\\
-\lVert u(t) \rVert, & t < 0 ,\\
\end{cases}
\]
is $C^\beta$ near $0$. 

For $t \geq 0$ sufficiently small, we have $\nu(t)=\lVert u(t) \rVert = \cos\left( \psi_f\left(\pi/2 - t\right) \right)$, and since $\psi_f(\pi/2 - t)$ is close to $\varepsilon \pi/2$, this gives 
\[
 \psi_f \left( \frac{\pi}{2} - t\right) = \varepsilon\arccos(\nu(t)).
\]
The right-hand side extends to a $C^\beta$ function for $t$ near $0$, so $\psi_f$ is $C^\beta$ at the north pole $\pi/2$. The same argument at the south pole, using $\eta(t) = (\sin(t) e_1, -\cos(t))$, shows that $\psi_f$ is also $C^\beta$ at $-\pi/2$. Thus $\psi_f$ is $C^\beta$. 

We now check that the assignment \eqref{eq: assignment tilde psi_f} defines the desired $\tilde\Pi_n$. Let $f \in \Centr_{\Diff^\beta(\RP^n)}(\tilde K)$. Since $f$ commutes with $\tilde K$, it maps $\tilde K$-orbits to $\tilde K$-orbits, and so it descends to a unique homeomorphism $\tilde \psi_f \colon [0, \pi/2] \to [0, \pi/2]$ such that $\tilde \lambda \circ f = \tilde\psi_f \circ \tilde\lambda$. Hence $f \mapsto \tilde \psi_f$ defines a homomorphism into $\Diff^0([0,\pi/2])$. It remains to prove that $\tilde\psi_f$ is $C^\beta$. To do this, choose a $C^\beta$-lift $F \colon \Sn \to \Sn$ of $f$, so that $\pi \circ F = f \circ \pi$. Such a lift always exists since $\Sn$ is simply connected, and it is a $C^\beta$-diffeomorphism since $\pi$ is a smooth covering map. Since $F$ is a lift of $f$, it is straightforward to verify that it commutes with the antipodal map $A$.  

Next, we claim that $F$ commutes with $K$. Let $k \in K$. Then since $f$ commutes with the projection of $k$ to $\tilde K$, it follows that the maps $F k$ and $k F$ are lifts of the same map. Thus they differ by either $\id_{\Sn}$ or the antipodal $A$, i.e.\ $F k = A^{\varepsilon(k)} k F$ for some $\varepsilon(k) \in \Z / 2\Z$. The function $\varepsilon \colon K \to \Z / 2\Z$ is continuous since the two possible lifts are separated by the discrete group $\{ \id_{\Sn}, A \} \cong \Z / 2\Z$. But $K \cong \SO(n)$ is a connected group, and $\varepsilon(1) = 0$. Hence $\varepsilon(k) = 0$ for all $k \in K$, so that $Fk = k F$ for all $k \in K$. Thus $F \in \Centr_{\Diff^\beta(\Sn)}(K)$.

Thus, we can apply the previous part, i.e.\ the assignment \eqref{eq: assignment psi_f} to $F$, giving a $C^\beta$-diffeomorphism $\psi_{F} \colon [-\pi/2, \pi/2] \to  [-\pi/2, \pi/2]$ satisfying $\lambda \circ F= \psi_{F} \circ \lambda$. Since $F$ commutes with $A$, we immediately see that $\psi_{F}$ must be odd. Indeed, choose $x \in \Sn$ with $\lambda(x) = s$, and write
\[
\psi_{F}(-s) = \psi_{F}(\lambda(Ax)) = \lambda(F(Ax)) = \lambda(AF(x)) = -\lambda(F(x)) = -\psi_{F}(s),
\]
for all $s \in [-\pi/2, \pi/2]$. 

The two possible lifts of $f$ are $F$ and $A F$, and their quotient maps on $[-\pi/2, \pi/2]$ are respectively $\psi_{F}$ and $-\psi_{F}$. We may hence choose the lift $F$ so that $\psi_{F}([0,\pi/2]) = [0, \pi/2]$. Since
\[
\tilde\lambda \circ \pi = |\lambda|, \qquad \tilde\lambda \circ f = \tilde\psi_f \circ \tilde\lambda, \quad \text{and} \quad \lambda \circ F = \psi_{F} \circ \lambda,
\]
it now follows that $\tilde\psi_f = \psi_{F} \mid_{[0,\pi/2]}$. Since $\psi_{F}$ is a $C^\beta$-diffeomorphism, it follows that $\tilde\psi_f$ is also a $C^\beta$-diffeomorphism. Thus the map $f \mapsto \tilde\psi_f$ is our desired homomorphism $\tilde \Pi_n$.

We now turn to the proofs of the abelianity of $\ker(\Pi_n)$  and $\ker(\tilde\Pi_n)$. Let $f \in \ker(\Pi_n)$. Since $\psi_f = \id_{[-\pi/2,\pi/2]}$, the map $f$ preserves every latitude. Hence, for every $s\in \left( -\pi/2, \pi/2\right)$ there is a map $f_s \colon \mathbb{S}^{n-1} \to \mathbb{S}^{n-1}$ such that $f(\cos(s) u, \sin(s)) = (\cos(s) f_s(u), \sin(s))$. Since $f$ commutes with $K$, we have $f_s(g\cdot u) = g \cdot f_s(u)$ for all $g \in K$. 

Suppose that $n \geq 3$. Take $e_1 \in \mathbb{S}^{n-1}$. Every element $g$ of the stabilizer $H\leq K$ of $e_1$ also fixes $f_s(e_1)$, since $f_s(e_1) = f_s(g \cdot e_1) = g \cdot f_s(e_1)$. But since the subgroup fixing $e_1$ acts by rotations on the orthogonal complement of $e_1$, the only two points of $\mathbb{S}^{n-1}$ fixed by this $H$ are $e_1$ and $-e_1$ (note that if $n=2$, then this conclusion fails). Hence $f_s(e_1) = \varepsilon(s) e_1$ for some $\varepsilon(s) \in \{ -1, 1\}$. 

For arbitrary $u \in \mathbb{S}^{n-1}$ choose some $g \in K$ with $g \cdot e_1 = u$ by transitivity. Then $f_s(u) = f_s(g\cdot e_1) = g \cdot f_s(e_1) = \varepsilon(s) u$. Since $f$ is continuous, the sign $\varepsilon(s) \in \{ -1, 1\}$ varies continuously with $s$, and since the interval $\left( -\pi/2, \pi/2\right)$ is connected, it follows that $\varepsilon(s)$ is constant. Thus either $f(v,t) = (v,t)$ or $f(v,t) = (-v,t)$ for all $(v,t) \in \Sn$. Thus $\ker(\Pi_n) \cong \Z / 2\Z$. 

Now assume that $n =2$. Each non-polar latitude is then $\mathbb{S}^1$, and hence $f_s \colon \mathbb{S}^1 \to \mathbb{S}^1$ commutes with every rotation. If $a_s := f_s(1) \in \mathbb{S}^1$, then for every $z \in \mathbb{S}^1$ we have $f_s(z) = f_s(z \cdot 1) = z \cdot f_s(1) = za_s$. Hence $f_s$ is also a rotation. Thus, for any $h \in \ker(\Pi_2)$, then on each latitude the restrictions of $f$ resp.\ $h$ are rotations, and so they commute. They also fix the two poles, so that $fh=hf$ on all of $\mathbb{S}^2$. Since $f, h \in \ker(\Pi_2)$ were arbitrary, it follows that $\ker(\Pi_2)$ is abelian.

Let $f \in \ker(\tilde\Pi_n)$. Choose the lift $F$ satisfying $\psi_{F}([0,\pi/2]) = [0, \pi/2]$ as above. Then $\tilde\psi_f = \psi_{F} \mid_{[0,\pi/2]}$, i.e.\ $\psi_{F}(s) = s$ for all $0 \leq s \leq \pi/2$. Since $\psi_{F}$ is odd, it follows that $\psi_{F}(s) = s$ also for all $-\pi/2 \leq s \leq \pi/2$. Thus $F \in \ker(\Pi_n)$. Since we just proved that this kernel is abelian, the fact that $\ker(\tilde\Pi_n)$ is also abelian now follows from the commutativity of its lifted elements. 

Finally, suppose that $n \geq 3$. We have just proved that $\ker(\Pi_n) = \{ \id_{\Sn}, J\}$, where $J(v,t) = (-v,t)$. Thus every element of $\ker(\tilde\Pi_n)$ is the projection of $\id_{\Sn}$ or $J$. The map $J$ commutes with the antipodal map $A$, and hence descends to an involution $\tilde J$ of $\RP^n$. This involution is non-trivial, centralizes $\tilde K$, and acts trivially on the quotient interval $[0, \pi/2]$. Thus if $n \geq 3$, we have $\ker(\tilde \Pi_n) = \{\id_{\RP^n}, \tilde J\} \cong \Z / 2\Z$, as desired. This finishes the proof of the proposition.\end{proof}

\section*{Acknowledgments} \noindent The authors wish to thank Sang-hyun Kim (KIAS) for helpful discussions and encouragement. 

\bibliographystyle{plain}
\bibliography{sn-critical-regularity.bib}

\end{document}